\documentclass[11pt]{article}
\usepackage{a4wide}

\usepackage[a4paper, margin=2.3cm]{geometry}
\usepackage{verbatim}
\usepackage{amsmath,amssymb,amsthm}
\usepackage{thmtools}
\usepackage{thm-restate}
\usepackage{color}
\usepackage{parskip}
\usepackage{hyperref}
\usepackage{parskip}
\usepackage{setspace}
\usepackage{graphicx}
\usepackage{mathtools}
\usepackage[dvipsnames]{xcolor}
\usepackage[thinc]{esdiff}
\usepackage{kotex}
\usepackage{enumitem}
 
\usepackage[capitalise,nameinlink]{cleveref}
\crefname{equation}{}{}
\hypersetup{
    colorlinks=true,
    linkcolor=blue,
    citecolor=magenta,      
    urlcolor=cyan,
}

\theoremstyle{plain}
\newtheorem{thm}{Theorem}[section]
\newtheorem{prop}[thm]{Proposition}
\newtheorem{cor}[thm]{Corollary}

\newtheorem{theorem*}{Theorem}[]

\theoremstyle{definition}

\theoremstyle{remark}

\newcommand{\defeq}{\vcentcolon=}

\newcommand{\F}{\mathbb{F}_{q}^{n}}

\newcommand{\FF}{\mathbb{F}}
\newcommand{\EE}{\mathbb{E}}

\renewcommand{\th}{\theta}
\newcommand{\G}{\Gamma}
\newcommand{\fqn}{\mathbb{F}_{\! q}^{^n}}

\renewcommand{\th}{\theta}

\newcommand{\f}{\mathcal{F}}

\newcommand{\bc}{\begin{center}}
\newcommand{\ec}{\end{center}}

\newcommand{\ra}{\rightarrow}

\newcommand{\el}{\end{lemma}}
\newcommand{\bl}[1]{\begin{lemma}\label{#1}}
\newcommand{\eco}{\end{corollary}}
\newcommand{\bco}[1]{\begin{corollary}\label{#1}}

\newcommand{\old}[1]{}

\old{

} 

\newcommand{\ga}{\alpha}

\newcommand{\raf}[1]{(\ref{#1})}
\newcommand{\beq}[1]{\begin{equation}\label{#1}}

\newcommand{\eeq}{\end{equation}}

\newcommand{\bean}{\begin{eqnarray*}}
\newcommand{\eean}{\end{eqnarray*}}
\newcommand{\bea}[1]{\begin{eqnarray}\label{#1}}
\newcommand{\eea}{\end{eqnarray}}

\newcommand{\mn}{\medskip\noindent }
\newcommand{\bb}{\Big}

\newcommand{\ee}{\EE}
\newcommand{\sfrac}[2]{\mbox{$\frac{#1}{#2}$}}

\begin{document}

\title{Threshold functions for incidence properties in finite vector spaces}
\author{ Jeong Han Kim \thanks{Korea Institute for Advanced Study (KIAS). Email: {\tt jhkim@kias.re.kr}} \and Ben Lund\thanks{Discrete Mathematics Group, Institute for Basic Science (IBS). Email: {\tt benlund@ibs.re.kr}} \and Thang Pham \thanks{University of Science, Vietnam National University, Hanoi. Email: {\tt thangpham.math@vnu.edu.vn}} \and Semin Yoo\thanks{Korea Institute for Advanced Study (KIAS). Email: {\tt syoo19@kias.re.kr}}}
\date{\today}

\maketitle

\begin{abstract}

The main purpose of this paper is to provide threshold functions for the events that a random subset of the points of a finite vector space has certain properties related to point-flat incidences.
Specifically, we consider the events that there is an $\ell$-rich $m$-flat with regard to a random set of points in $\FF_q^n$, the event that a random set of points is an $m$-blocking set, and the event that there is an incidence between a random set of points and a random set of $m$-flats. One of our key ingredients is a stronger version of a recent result obtained by Chen and Greenhill (2021).
\end{abstract}

\section{Introduction}

A fundamental problem in the theory of random sets is that of determining thresholds for monotone properties.
Thresholds were first discovered by Erd\H{o}s and Renyi in the context of random graphs \cite{ER1960}.
Since then, thousands of papers have been written on the subject, both in specific contexts such as for random graphs or random sets of integers, and in the general setting of random subsets of an arbitrary finite set.

Here, we study thresholds for properties of point sets in finite vector spaces.
Hundreds of papers have been written on extremal questions in finite geometry, but we are only aware of a single work on thresholds in this context, by Chen and Greenhill \cite{chen}.
Our contributions are to improve the main result of Chen and Greenhill, and to study thresholds for several properties defined by the incidences between points and affine subspaces.

In this paper, we are concerned with thresholds for local properties, defined as follows.
Let $U_1, U_2, \ldots$ be a sequence of finite sets, with $|U_{i+1}|>|U_i|$ for all $i \geq 1$.
For each $i$, let $\f_i$ be a family of subsets of $U_i$, each of size $\ell_i$, and let $X_i$ be a random subset of $U_i$, such that each element of $U_i$ is in $X_i$ with probability $\theta_i$, independently of all other elements of $U_i$.
Let $\Gamma = \Gamma_i$ be the event that $X_i$ contains at least one element of $\f_i$.
A function $t=t(i)$ is a {\em threshold} for $\Gamma$ if
\begin{enumerate}
    \item $\Pr[\Gamma] \rightarrow 0$ if $\theta = o(t)$, and
    \item $\Pr[\Gamma] \rightarrow 1$ if $\theta = \omega(t)$.
\end{enumerate}
A function $t$ has the stronger property of being a {\em sharp threshold} for $\Gamma$ if, for every $\epsilon > 0$,
\begin{enumerate}
    \item $\Pr[\Gamma] \rightarrow 0$ if $\theta \leq (1-\epsilon)t$, and
    \item $\Pr[\Gamma] \rightarrow 1$ if $\theta \geq (1 + \epsilon)t$.
\end{enumerate}
Clearly, if $X$ contains an element of $\f$ and $X \subset Y$, then $Y$ contains an element of $\f$; hence, the event $\Gamma$ describes a {\em monotone} property of sets.
Bollob\'as and Thomason \cite{bolo2} showed that every monotone property has a threshold function.

Chen and Greenhill proved the following result on threshold functions for properties of point sets in finite vector spaces, and used it to derive threshold functions for several specific geometrically defined properties.
\begin{thm}[Chen-Greenhill, \cite{chen}]\label{cf-chengreenhill}
Let $\mathcal{F}$ be a family of $\ell$-subsets of $\mathbb{F}_q^n$, and $X$ be a random subset chosen as above. Suppose there exist integer constants $b>0$, $c\ge 0$ such that $|\f|=\Theta (q^{bn-c})$. 
Assume that for any set  $S$  of $k$ distinct points in $\mathbb{F}_{q}^{n}$, the number of elements of $\mathcal{F}$ which contain~$S$ is 
\begin{equation}\label{condition3}
  \begin{cases}
   O(q^{(b-k)n-c}) & \text{if $1\leq k\leq b-1$}, \\
   O(1) &\text{otherwise}.
  \end{cases}
\end{equation}

Then the event $\Gamma=\Gamma(\mathcal{F}, \theta)$ that there is a subset of $X$ belonging to $\mathcal{F}$ has a threshold function $q^{(c-bn)/\ell}$, as $q+n\to \infty$.
\end{thm}

As consequences, they obtained threshold functions for several precise geometric structures;
for examples, non-trivial $k$-term arithmetic progressions, non-trivial parallelograms, and non-trivial right angles. We refer the interested reader to \cite{chen} for more details. 

Our first result strengthens \cref{cf-chengreenhill}.
Although we are primarily interested in finite vector spaces, \cref{new} applies more generally.

Let $U$ be a (finite) set and $X$ be the random subset  of $U$ such that each element of $U$ is in $X$ with probability $\th$, independently of all other elements of $U$. 
For an integer $\ell$ with $1\leq \ell < |U|$ 
and a family $\f $ of $\ell$-subsets of $U$, we are interested in the event $\G=\G(\f;\th)$ that there is a subset of $X$ belonging to $\f$. 

\begin{restatable}{prop}{new}
\label{new}
Let $U$ be a finite set, $\ell$ an integer with $1 \leq \ell < |U|$, and $\f$ a family of $\ell$-subsets of $U$, with $|\f|$ and $|U|$ tending to infinity.
If  there exists a constant $c>0$ satisfying  
\begin{equation}\label{new-condition}
\sum_{k=1}^{\ell-1}c^{k}|\f|^{k/\ell}|I_{k}|~ \leq ~|\f|^{2}, 
\end{equation}
where $I_{k}=I_{k}(\f):=\{(S,T)\in \f \times \f:|S \cap T|=k \}$, then the event $\G=\G(\f;\th)$ has a threshold function $|\f|^{-1/\ell}$ 
as $|\f|\to \infty$.
\end{restatable}

When $\ell$ is bounded, we have the same conclusion of \cref{new} with a simpler condition.
\begin{restatable}{prop}{revision}
\label{thm:revision}
Let $\ell=O(1)$. Assume that for any subset $R$ of $U$ with $|R|=k$ for each $1 \leq k \leq \ell-1$, the number of elements of $\f$ which contain $R$ is $O(|\f|^{1-k/\ell})$. Then the event $\G=\G(\f;\th)$ has a threshold function $|\f|^{-1/\ell}$ as $|\f|\to \infty$.
\end{restatable}

\cref{thm:revision} has appeared in the literature before, for example see \cite[Lemma 4.1]{CITE}. It is not hard to see that \cref{cf-chengreenhill} follows from \cref{thm:revision}, as the hypothesis of the \cref{thm:revision} is weaker.

The hypothesis of \cref{new} is not needed to show that $\G$ has a threshold function $t \geq |\f|^{-1/\ell}$.
Indeed, if $\th = o(|\f|^{-1/\ell})$, then the expected number of elements of $\f$ contained in $X$ is $o(1)$.
Hence, by Markov's inequality, $\Pr[\Gamma] = o(1)$.
On the other hand, without the hypothesis of \cref{new}, it is not always the case that $|\f|^{-1/\ell}$ is a threshold function.
For example, if $\f$ is the family of sets of size $\ell$ that contain a fixed element $x \in U$, then $\Pr[\Gamma] < \Pr[x \in X] = \theta$.
Hence, $\Pr[\G] > 1-o(1)$ only if $\theta > 1-o(1)$, and in particular $|\f|^{-1/\ell}$ is not a threshold function.

Recently, Frankston, Kahn, Narayanan, and Park \cite{CITE} proved that a more sophisticated unconditional lower bound on the threshold function is always tight, up to a logarithmic factor.
Even though it is simpler and less general than the bound of Frankston et. al., \cref{new} is well-suited to our applications.

For some applications, we need a variant of \cref{new} (see \cref{th:2MomentThreshold}) that can sometimes give sharper tail bounds when $\ell$ is large.

Our main results are thresholds for several geometric properties related to incidences between points and flats ({\em i.e.} affine subspaces) that have previously been studied from an extremal perspective.

An {\em $m$-flat} is a translate of an $m$-dimensional linear subspace.
An $m$-flat $L$ is {\em $\ell$-rich with regard to a set $X$} if $|L \cap X| \geq \ell$.
It is known that, for any $\epsilon > 0$ and $q \rightarrow \infty$, if $X \subset \FF_q^n$ with $|X| = (1+\epsilon)(\ell-1)q^{n-m}$, then a constant proportion of all $m$-flats are $\ell$-rich \cite{ls}.
We give threshold functions for event that there is at least one $\ell$-rich $m$-flat in a random set of points in $\FF_q^n$, under various assumptions on $m,n,q$ and $\ell$.

\cref{new} applies directly to give the following threshold function for $\ell$-rich lines.

\begin{restatable}{thm}{lcol}
\label{l-col}
Suppose $n\geq 2$ is an integer and  $q$ is a prime power with   $n+q \ra \infty$.
 Let $X$ be  the random subset 
of  $\mathbb{F}_{\! q}^{n} $ 
such that each element of $\mathbb{F}_{q}^{n}$ is
 in $X$ with probability $\th$, independently of all other elements. 
 Then, for an integer $\ell$ satisfying $3 \leq  \ell \leq q $ 
 and  $\ell= O(n \ln q)$, 
 the event   $\G=\G(\th;\ell)$ that there is an $\ell$-rich line with respect to $X$ has a threshold function $(q^{(2n-2)} {q \choose \ell})^{-1/\ell}$. 
\end{restatable}
Note that Chen and Greenhill's \cref{cf-chengreenhill} immediately implies the case $\ell=3$ of \cref{l-col}, but the hypothesis of \cref{cf-chengreenhill} is not satisfied for $\ell \geq 4$.
\cref{new} does not yield \cref{l-col} for $\ell = \omega(n \log q)$, but for this case a direct application of a Chernoff bound yields the following result.

\begin{restatable}{thm}{lrichsharp}\label{lrichsharp}
Suppose $n,m$  are integers with $1 \leq m < n$, and $q$ is a prime power, with $n+q \rightarrow \infty$.
Let $X \subseteq \mathbb{F}_q^n$ such that each element of $\mathbb{F}_q^n$ is in $X$ with probability $\th$, independently of all other elements.
Then, for an integer $\ell \leq q^m$ with $\ell = \omega((n-m)m \log(q))$, the event $\Gamma= \Gamma(\th;\ell)$ that there is an $\ell$-rich $m$-flat with respect to $X$ has a sharp threshold $\ell q^{-m}$.
\end{restatable}

Note that the threshold $\ell q^{-m}$ of \cref{lrichsharp} matches the threshold $(q^{2n-2} \binom{q}{\ell})^{-1/\ell}$ of \cref{l-col} up to a constant factor for $\ell = \Theta(n \log q)$.

\cref{new} and \cref{thm:revision} do not apply directly to give a threshold function for $\ell$-rich $m$-flats when $m \geq 2$.
For example, suppose that $n=3$, $m=2$ and $\ell = 4$, and let $\f$ be the family of sets of $4$ coplanar points in $\FF_q^3$.
Let $R$ be three collinear points.
Then, the number of elements of $\f$ that contain $R$ is $q^3-3 = \omega(q^{11/4})$, and hence the hypothesis of \cref{thm:revision} is not satisfied.
Similar considerations hold for \cref{new}.

Although a direct application of \cref{new} does not apply for $m \geq 2$, we do manage to generalize \cref{l-col} by a more complicated argument.
In particular, \cref{new} applies to give a threshold for a large sub-family of the family of $\ell$-sets contained in $m$-flats.
A separate argument derives the following tight threshold result from the threshold for the sub-family.

\begin{restatable}{thm}{mflat}
\label{m-flat}
Suppose $n \geq 3$ and $q$ is a prime power with $q \to \infty$ \textup{(}$n$ could be bounded or not\textup{)}. Let $X$ be the random subset of $\mathbb{F}_{q}^{n}$ such that each element of $\mathbb{F}_{q}^{n}$ is in $X$ with probability $\theta$, independently of all other elements. Let $\ell$ and $m$ be bounded with $m \ge 2$, and let $\f$ be the family of $\ell$-rich $m$-flats. Then the event $\G=\G(\theta; \ell)$ that $X$ contains an $\ell$-rich $m$-flat in $\mathbb{F}_{q}^{n}$ has a threshold function $(q^{(n-m)(m+1)+m\ell})^{-1/\ell}$.
\end{restatable}

\cref{m-flat} generalizes \cref{l-col} for $m=O(1)$ under the extra conditions that $q \rightarrow \infty$ and $\ell = O(1)$.
We believe the same bound should hold even without these extra conditions, even for unbounded $m$.
We give a separate proof for the case that $n \rightarrow \infty$ and $\ell = q^m$.
Note that an $m$-flat $L$ is $q^m$-rich with respect to $X$ if and only if $L$ is completely contained in $X$.

\begin{restatable}{thm}{mplanes}\label{m-planes}
\label{m-planes}
Let $n$ be an integer with $n \rightarrow \infty$, and let $q$ be a prime power (bounded or not).
 Let $X$ be the random subset of $\F$ such that each element of $\F$ is in $X$ with probability $\th$, independently of all other elements.
 Let $m$ be bounded, and let $\f$ be the family of $m$-flats in $\F$. 
 Then, the event $\G=\G(\f;\th)$ that $X$ contains an $m$-flat has a threshold function $q^{-n(m+1)q^{-m}}$.
 If $q \rightarrow \infty$, then this threshold is sharp.
\end{restatable}

\cref{m-planes} is most interesting when $q = o((n\log(n))^{1/m})$, since otherwise $q^{-n(m+1)q^{-m}}$ is bounded from below by a constant.
Chen and Greenhill proved \cref{m-planes} in the case that $q$ is bounded; our contribution is to prove the sharp threshold for the case that $q = \omega(1)$ and $q=o((n\log(n))^{1/m})$.

Next, we give a threshold function for the event that $X$ is a blocking set.
An {\em $m$-blocking set} in $\FF_q^n$ is a set of points that intersects every $m$-flat.
Several important questions in finite geometry amount to determining the largest or smallest possible size of a minimal blocking set.
For example, the celebrated cap set problem \cite{CLP, capset} asks for the largest size of a set in $\FF_3^n$ that does not contain a line.
Considering the complement of such a set, this is  equivalent to asking for the size of the smallest $1$-blocking set in $\FF_3^n$.
Blocking sets are also very closely related to Nikodym sets, which have been studied in the context of Kakeya sets; see \cite{lsw} for a detailed discussion.

More is known about $1$-blocking sets in $\FF_q^2$ than for other values of $m$ and $n$, although many open problems remain even in this case \cite{bszsz}.
An observation in this setting ($m=1,n=2$) is that it is relatively easy to construct minimal blocking sets of size $\Theta(q \log q)$, but larger and smaller minimal blocking sets are comparatively rare \cite{szgw}.
Here, we formalize this observation in higher dimensions by giving a threshold.

\begin{restatable}{thm}{blockingSetThreshold}
\label{th:blockingSetThreshold}
Let $m, n$ be integers with $1 \leq m < n/2$, and $q$ be a prime power, with $q$ tending to infinity and $m$ be a constant, and $n=o(q^m (\log q)^{-1})$.
Let $X$ be a set of points in $\FF_q^n$ such that each point of $\FF_q^n$ is in $X$ with probability $\th$, independently of all other points.
Then, the event that $X$ is an $m$-blocking set has a threshold $nq^{-m}\log(q)$.
\end{restatable}

The assumption $m<n/2$ arises in our proof from the fact that most pairs of $m$-flats intersect when $m \geq n/2$.
It would be very interesting to remove this assumption.

Next, we take $X$ to be a random set of points and $m$-flats, and give a threshold for the event that some point in $X$ is incident to an $m$-flat in $X$.
Let $P$ be a set of points and $H$ be a set of hyperplanes in $\mathbb{F}_q^n$, and denote by $I(P,H) = |\{(p,h) \in P \times H : p \in h\}|$ the number of incidences between $P$ and $H$.
It was first shown by Haemmers (see \cite{WW} or \cite{vinh2}) that
\[\left\vert I(P, H)-\frac{|P||H|}{q} \right\vert \ll q^{\frac{n-1}{2}}\sqrt{|P||H|}.\]
It follows that if $|P||H|\gg q^{n+1}$, then there always exists an incidence between $P$ and $H$. 

Building on Hoi's work \cite{hoi}, Vinh \cite{vinh1} proved a version for a random sets, which states that for any $\alpha>0$, there exists an integer $q_0=q(\alpha)$ and a positive number $C_{\alpha}$ with the following property. 
When a point set $P$ and a hyperplane set $H$ where $|P|=|H|=s\ge C_{\alpha}q$ are chosen randomly in $\mathbb{F}_q^d$, the probability that $I(P, H)=0$ is at most $\alpha^s$ provided that $q\ge q_0$.

In the next theorem, we provide a more general threshold variant of this result.

\begin{restatable}{thm}{incidenceThreshold}
\label{th:incidenceThreshold}
Let $m, n$ be integers with $1 \leq m < n$ and $q$ a prime power, with $n+q$ tending to infinity.
Let $X$ be a set of points in $\FF_q^n$ such that each point is in $X$ with probability $\th$ independently of all other points, and let $Y$ be a set of $m$-flats in $\FF_q^n$, such that each $m$-flat is in $Y$ with probability $\th$, independently of all other $m$-flats.
Let $\G$ be the event that there is a pair $(x,y) \in X \times Y$ such that $x$ is contained in $y$.
If $m \in \{1,n-1\}$, then the threshold for $\G$ is $q^{-(1/2)((n-m)(m+1) + m)}$.
Otherwise, the threshold for $\G$ is $q^{-n}$.
\end{restatable}

\section{Thresholds for general set families}

In this section, we prove \cref{new} and \cref{thm:revision}.
We first prove \cref{th:2MomentThreshold}, which gives somewhat better tail bounds in certain applications.
Then we derive \cref{new} as a corollary to \cref{th:2MomentThreshold}, and finally derive \cref{thm:revision} as a corollary to \cref{new}.

\begin{prop}\label{th:2MomentThreshold}
Let $U$ be a finite set with $1 \leq \ell \leq |U|$, and let $\mathcal{F}$ be a family of $\ell$-subsets of $U$. Let $X$ be  the random subset of  $U$ such that each element of $U$ is in $X$ with probability $\th$, independently of all other elements. For any $0<c,c'<1$ and $\theta = c^{-1}|\f|^{-1/\ell}$, if
\begin{equation}\label{eq:2condition}
\sum_{k=1}^{\ell-1} c^k |\f|^{k/\ell}|I_k| \leq c'|\f|^2\end{equation}
where $I_{k}=I_{k}(\f):=\{(S,T)\in \f \times \f:|S \cap T|=k \}$, then $X$ contains an element of $\mathcal{F}$ with probability at least $1 - c' - c^\ell$.
\end{prop}

\begin{proof}
Let $Y=Y(\f)$ be the number of elements of $\f$ contained in $X$.
The expectation of $Y$ is 
\begin{equation}\label{expectation}
\mathbb{E}[Y]=\mathbb{E}\Big[\sum_{S \in \f}\mathbf{1}(S \subset X)\Big]=\sum_{S \in \f}\mathbb{E}[\mathbf{1}(S \subset X)]=\sum_{S \in \f} \Pr[S \subset X]=\th^{\ell}|\f|= c^{-\ell},
\end{equation}
and the second moment is
$$
\mathbb{E}[Y^{2}]=\mathbb{E}\left[\sum_{S \in \f}\sum_{T \in \f}\mathbf{1}(S \subset X)\mathbf{1}(T \subset X) \right]
=\sum_{k=0}^{\ell}\sum_{S,T \in \f \atop |S \cap T|=k}\textup{Pr}[S \cup T \subset X].
$$
Since $\textup{Pr}[S \cup T \subset X]=\th^{2\ell-k}$ for $|S\cap T|=k$, we have 
\begin{equation}\label{sm} \mathbb{E}[Y^{2}]=\sum_{k=0}^{\ell}\sum_{S,T \in \f \atop |S \cap T|=k}\th^{2\ell-k}=\sum_{S,T \in \f \atop |S \cap T|=0}\th^{2\ell}+
\sum_{S,T \in \f \atop |S \cap T|=\ell}\th^{\ell}
+\sum_{k=1}^{\ell-1}\th^{2\ell-k} |I_k|.
\end{equation}
Notice that  
$$ 
\sum_{S,T \in \f \atop |S \cap T|=0}\th^{2\ell} \leq \th^{2\ell}|\f|^2 =\ee[Y]^2, ~~~{\rm and} ~~~
\sum_{S,T \in \f \atop |S \cap T|=\ell}\th^{\ell} = \th^{\ell}|\f| =\ee[Y]. $$

We use the hypotheses of the theorem to bound the third term:
$$
\sum_{k=1}^{\ell-1} \th^{2\ell-k} |I_k| ~
\leq ~ \th^{2\ell} \sum_{k=1}^{\ell-1} c^{k} |\f|^{k/\ell}|I_k| ~\leq~ c'\th^{2\ell}|\f|^2 = c'\mathbb{E}[Y]^2 .$$

Combining, \cref{sm} simplifies to
\[ \mathbb{E}[Y^2] = (1+c')\mathbb{E}[Y]^2 + \mathbb{E}[Y].\]

Using Chebyshev's inequality (see Lemma 5.1.4 of \cite{MV}) and the bound \cref{expectation} on $\mathbb{E}[Y]$,
$$ 
\Pr [ Y=0] 
\leq \Pr \bb[ |Y-\ee[Y]|\geq \ee[Y]\bb]
\leq \frac{\ee[Y^2]-\mathbb{E}[Y]^2}{\mathbb{E}[Y]^{2}}\leq \frac{1}{\mathbb{E}[Y]}+c' \leq c^\ell + c'.$$
This completes the proof.
\end{proof}

\new*

\begin{proof}
We need to prove that $\Pr[\G] \to 0$ if $\th=o(|\f|^{-1/\ell})$, and $\Pr[\G] \to 1$ if $\th=\omega(|\f|^{-1/\ell})$.
Let $Y=Y(\f)$ be the number of subsets of $X$ that are elements of $\f$. Then $\G$ is the event that $Y\geq 1$.

First, suppose that $\th=o(|\f|^{-1/\ell})$. Then  
\[
\mathbb{E}[Y]=\mathbb{E}\Big[\sum_{S \in \f}\mathbf{1}(S \subset X)\Big]=\sum_{S \in \f}\mathbb{E}[\mathbf{1}(S \subset X)]=\sum_{S \in \f} \Pr[S \subset X]=\th^{\ell}|\f|.
\]
Markov's inequality and  $\th=o(|\f|^{-1/\ell})$  imply that 
$$
\textup{Pr}[\G] = \Pr[ Y \geq 1]\leq \mathbb{E}[Y] \ra 0 , $$
as desired. 

\mn 

Suppose now that $\th = \omega(|\f|^{-1/\ell})$, 
and that \cref{new-condition} holds for the constant $c_1$.
Define $t$ so that $\th= t^{-1} |\f|^{-1/\ell}$, and note that $t = o(1)$.
For $|\f|$ large enough, we have $t < c_1$ and hence
\[\sum_{k=1}^{\ell -1} t^k |\f|^{k/\ell} |I_k| = \left(\frac{t}{c_1} \right) \sum_{k=1}^{\ell - 1} c_1 t^{k-1} |\f|^{k/\ell} |I_k| \leq \left(\frac{t}{c_1} \right) \sum_{k=1}^{\ell - 1} c_1^k |\f|^{k/\ell} |I_k| \leq \left(\frac{t}{c_1} \right) |\f|^2.\]
Now \cref{th:2MomentThreshold} implies that $\Pr[\G] \geq 1 - tc_1^{-1} - t^\ell \ra 1$, as desired.
\end{proof}

\revision*

\begin{proof}
It suffices to show that the assumption implies \cref{new-condition}. Notice that $I_{k}=I_{k}(\f)$ satisfies
\begin{equation}
|I_{k}|
= \sum_{S \in \f } \sum_{T\in \f} \mathbf{1} (|S\cap T|=k)
\leq  \sum_{S \in \f } \sum_{R\subset S \atop |R|=k} \sum_{T\in \f }\mathbf{1} (R\subset T).
\end{equation}
Since the hypothesis means  $\sum_{T\in \f }\mathbf{1} (R\subset T)= O(|\f|^{1-k/\ell})$, there is a constant $\alpha >0$ such that 
$$ |I_{k}|
\leq \ga \sum_{S \in \f } {\ell \choose k} |\f|^{1-k/\ell}
=  \ga {\ell \choose k} \, |\f|^{2-k/\ell}. $$
Thus
$$
|\f|^{k/\ell}|I_{k}|
\leq  \ga  \, {\ell \choose k}|\f|^{2},
$$
which yields, for $ c:= \frac{ \ln (1+\ga^{-1})}{\ell}$, 
\[\sum_{k=1}^{\ell-1}c^{k}|\f|^{k/\ell}|I_{k}| 
\leq 
\ga 
\sum_{k=1}^{\ell-1}
{\ell \choose k} c^k |\f|^{2}
= \ga 
\bb((1+c)^\ell -1-c^{\ell} \bb)|\f|^{2} \leq \ga 
(e^{c\ell} -1 )|\f|^{2} 
=|\f|^{2} 
. 
\]
This completes the proof.
\end{proof}

\section{Rich flats}

In this section, we prove \cref{l-col}, \cref{lrichsharp}, and \cref{m-flat}.
We use the following elementary facts about the number of $m$-flats in $\FF_q^n$.

An $m$-flat is a translate of an $m$-dimensional linear subspace.
The number of $m$-dimensional linear subspaces in $\F$ is
\[\binom{n}{m}_q \defeq \frac{(q^n-1)(q^n-q)\ldots(q^n-q^{m-1})}{(q^m-1)(q^{m}-q)\cdots(q^m-q^{m-1})}. \]
Indeed, the number of $m$-tuples of linearly independent vectors in $\F$ is $(q^n-1)(q^n-q)\ldots(q^n-q^{m-1})$, and the number of $m$-dimensional subspaces is the number of $m$-tuples of linearly independent vectors in $\F$ divided by the number of linearly independent vectors in $\mathbb{F}_q^m$.
The number of $m$-flats is $q^{n-m} \binom{n}{m}_q$, since there are $q^{n-m}$ distinct translates of a fixed $m$-dimensional subspace.
Similarly, for $0 \leq d < m$, the number of $m$-flats that contain a fixed $d$-flat is $\binom{n-d}{m-d}_{q}$.

For $q \rightarrow \infty$, we have the estimate $\binom{n}{m}_q = (1+o(1))q^{(n-m)m}$.
For general $q \geq 2$, we have the bounds 
\begin{equation}\label{eq:qBinBound}
q^{(n-m)m} \leq \binom{n}{m}_q < 4q^{(n-m)m}.\end{equation}

\lcol*
 
\begin{proof}
Let $\f$ be the family of all $\ell$-rich lines in $\mathbb{F}_{q}^{n}$. Since the number of lines in $\mathbb{F}_{\! q}^{n}$ is $\binom{q^{n}}{2}/\binom{q}{2}$ and each line contains ${q \choose \ell }$ sets of $\ell$ points, we have that 
\beq{ff}
|\f|= 
\frac{q^{n-1}(q^n-1)}{q-1}
\binom{q}{\ell}
=
\frac{(1-q^{-n})q^{2n-2}}{1-q^{-1}}
\binom{q}{\ell},
~~~~\mbox{in particular} ~~~~ 
|\f| = \Theta \bb(  q^{(2n-2)}
{ q \choose \ell} 
\bb) . 
\eeq

If $\ell$ is bounded, then $|\f| =\Theta( q^{2n+\ell -2})$ and we may apply \cref{thm:revision}.
Notice that  
the number of lines containing a fixed point is $(q^{n}-1)/(q-1)$, and hence the number of collinear $\ell$-sets containing a fixed point is  
$$\frac{q^{n}-1}{q-1} {q-1 \choose \ell-1}
=O( q^{n+\ell-2} ) . $$
As $\phi (n) := 
(2n+\ell -2) (1-1/\ell) - ( n+\ell-2)$ is increasing as $n$ increases and  $\phi(1) = 0$, 
we have that 
$$\frac{q^{n}-1}{q-1} {q-1 \choose \ell-1}  = O( q^{(2n+\ell -2) (1-1/\ell)}) = O( |\f|^{1-1/\ell}). $$
Similarly, for $k$ with $2 \leq k \leq \ell-1$ and a fixed $k$-subset $R$ of $\fqn$, the number of collinear $\ell$-sets containing $R$ is at most 
$$ {q-k  \choose \ell- k } = O(q^{\ell-k})= O(q^{(2n+\ell -2) (1-k/\ell)}) = O( |\f|^{1-k/\ell}) ,  $$
where the maximum occurs when the points of $R$ are collinear. 

Suppose $\ell$ is not bounded but $\ell= O(1+ (2n-2)\ln q)$.
In particular, $n\geq 2$ and   $\ell$ and $q$ ($\geq \ell$)  go to infinity.
Then, for $I_k := \{ (S,T) \in \f \times \f : | S\cap T| =k \}$, similar arguments as above give $$  |I_1| = |\f| {\ell \choose 1}  \bb( \frac{q^n-1}{q-1} -1\bb) {q-1 \choose \ell-1}+ |\f| {\ell \choose 1} {q-\ell  \choose \ell-1}
= (1+o(1))|\f| \ell q^{n-1} {q-1 \choose \ell-1}, $$
and 
$$
|I_k|= |\f| {\ell \choose k } {q-\ell 
\choose \ell -k } \leq |\f| {\ell \choose k } {q-k 
\choose \ell -k } ~~~~~\mbox{for $2\leq k \leq \ell-1$}. 
$$
 As $\ell$ and $q$ are large enough, we may use ${ q\choose \ell} \leq \frac{1}{2} (\frac{eq}{\ell})^{\ell}$ 
and 
\beq{fbound} |\f| \leq  q^{2n-2} \bb(\frac{eq}{\ell}\bb)^{\ell}\eeq
to obtain 
\bean 
\frac{|\f |^{1/\ell}|I_1|}{|\f |} 
&\leq &  (1+o(1)) |\f |^{1/\ell} \ell q^{n-1} {q-1 \choose \ell-1}  \\
&=& (e+o(1)) \ell q^{n-1+(2n-2)/\ell} \bb(\frac{q}{\ell}\bb)
{q-1 \choose \ell-1}   \\
&=& (e+o(1)) \ell  q^{n-1+(2n-2)/\ell}  {q \choose \ell} . 
\eean 
If $n\geq 3$, then 
$$\ell q^{n-1 + (2n-2)/\ell} 
\leq q^{n + (2n-2)/\ell} = 
o( q^{2n-2}). $$
If $n=2$, 
 $$ \ell q^{(2n-2)/\ell}
 = \ell q^{2/\ell}
 \leq q^{0.2}  + (1+o(1))\ell  \leq 2q  $$ 
 and hence 
 $\ell q^{n-1 + (2n-2)/\ell}  
 \leq 2q^2 = 2q^{2n-2}. $
 Therefore, 
\beq{i1}
 \frac{|\f |^{1/\ell}|I_1|}{|\f |} 
 \leq 6 q^{2n-2} {q \choose \ell} 
 \leq 7 |\f| , \eeq
 as 
 $ |\f|=
\frac{(1-q^{-n})q^{2n-2}}{1-q^{-1}}
\binom{q}{\ell}
= (1+o(1))q^{2n-2}
\binom{q}{\ell}. $

Suppose $k\geq 2$, then \cref{fbound} gives 
\bean 
\frac{|\f |^{k/\ell}|I_k|}{|\f |} 
&\leq &  q^{k(2n-2)/\ell} \bb( \frac{eq}{\ell}\bb)^k  {\ell \choose  k} {q-k \choose \ell -k }
= (e q^{(2n-2)/\ell})^k     {\ell \choose  k} \bb(\frac{q}{\ell}\bb)^k{q-k \choose \ell -k } . 
\eean 
Since $$
\bb(\frac{q}{\ell}\bb)^k{q-k \choose \ell -k }
\leq  \frac{q^k} {q(q-1) \cdots (q-k+1)}   {q \choose \ell} 
$$
and 
$$
\frac{q^k} {q(q-1) \cdots (q-k+1)} 
= \frac{q^k (q-k)!} {q!} 
\leq  \frac{c_{_1} q^k (\frac{q-k}{e})^{q-k}}{(\frac{q}{e})^{q}} \leq \frac{c_{_1}  q^k (\frac{q}{e})^{q-k}}{(\frac{q}{e})^{q}} =c_{_1}  e^k ,
$$
for a positive constant $c_{_1}$,
it follows  that
$$
\frac{|\f |^{k/\ell}|I_k|}{|\f |}
\leq c_{_1} {q \choose \ell } {\ell \choose  k}(e^2 q^{(2n-2)/\ell})^k.
$$
For any constant $c>0$, 
$$ \sum_{k=2}^{\ell-1} c^k |\f |^{k/\ell}|I_k|
\leq  
c_{_1}  {q \choose \ell }
\sum_{k=2}^{\ell-1} {\ell \choose  k} (c \, e^2 q^{(2n-2)/\ell})^k   
\leq c_{_1} \bb(1+ c \, e^2 q^{(2n-2)/\ell}\bb)^\ell
 {q \choose \ell }. 
$$
If 
$ q^{(2n-2)/\ell} \leq 3$, then 
$$
\bb(1+ c \, e^2 q^{(2n-2)/\ell}\bb)^\ell
\leq \bb(1+ 3\, c \, e^2 \bb)^\ell, $$
and if $q^{(2n-2)/\ell} > 3$ then 
\[1+ce^{2}q^{(2n-2)/\ell} \leq \frac{q^{(2n-2)/\ell}}{3}+ce^{2}q^{(2n-2)/\ell}=\frac{(1+3ce^{2})q^{(2n-2)/\ell}}{3},\]
and hence
$$\bb(1+ c \, e^2 q^{(2n-2)/\ell}\bb)^\ell
\leq \bb( \frac{1+3c\, e^2}{3} \bb)^{\ell} q^{2n-2}. $$
Thus, $$ 
\bb(1+ c \, e^2 q^{(2n-2)/\ell}\bb)^\ell 
\leq \bb(1+3c\, e^2 \bb)^{\ell}+   
\bb( \frac{1+3c\, e^2}{3} \bb)^{\ell} q^{2n-2}. 
$$
Taking a (small) constant $c>0$ satisfying 
$$ 2 \ell \ln (1+ 3\,c \, e^2 ) \leq (2n-2) \ln q $$ 
(recall $\ell= O( (2n-2)\ln q)$) and $c\, e^2 \leq \sfrac{1}{3}$, 
we have that 
$$
3c_{_1} \bb(1+ c \, e^2 q^{(2n-2)/\ell}\bb)^\ell
\leq q^{2n-2},
$$
and hence 
$$
\sum_{k=2}^{\ell-1} c^k |\f |^{k/\ell}|I_k|
\leq  \frac{q^{2n-2}}{3} 
{q \choose \ell} 
\leq \sfrac{1}{2}|\f|^2. 
$$
This together with 
\cref{i1} yields 
$$
\sum_{k=1}^{\ell-1} c^k |\f |^{k/\ell}|I_k|
\leq \bb(\frac{7}{3 e^2} +\frac{1}{2}\bb) |\f|^2
\leq |\f|^2,
$$
as desired. 
\end{proof}

The proof of \cref{lrichsharp} depends on the following Chernoff bounds for the sum of independent random variables.

Let $p \in [0,1]$.
Let $X_1, X_2, \ldots, X_n$ be mutually independent with $\Pr[X_i = 1] = p$ and $\Pr[X_i = 0] = 1-p$.
Let $X = X_1 + X_2 + \ldots + X_n$.

We denote $e^x$ by $\exp(x)$.

A special case of Theorem A.1.15 in ``The Probabilistic Method" by Alon and Spencer \cite{alon} is that
\begin{equation}\label{eq:lowerChernoff}
\Pr[X < (1-\epsilon)pn] < \exp(-\epsilon^2 pn / 2) \end{equation}
for any $\epsilon \in (0,1)$.

Another special case of Theorem A.1.15 is
\begin{equation}\label{eq:upperChernoff}
\Pr[X > (1+\epsilon)pn] < \exp((\epsilon - (1+\epsilon) \ln(1+\epsilon))pn)\end{equation}
for any $\epsilon > 0$.
We use the inequality $\ln(1+\epsilon) \geq 2\epsilon/(2 + \epsilon)$, valid for $\epsilon \geq 0$, to obtain
\begin{equation}\label{eq:simplifiedUpperChernoff}
\Pr[X > (1+\epsilon)pn] < \exp(-\epsilon^2pn/(2+\epsilon))
\end{equation}
from \cref{eq:upperChernoff}.

\lrichsharp*

\begin{proof}

Let $Y$ be the number of $m$-flats in $\mathbb{F}_q^n$ that each contain at least $\ell$ points of $X$, and let $\mathcal{L}$ be the set of $m$-flats in $\mathbb{F}_q^n$.

For $\epsilon \in (0,1)$, suppose that $\theta = \ell q^{-m} (1+\epsilon)^{-1} = (1-\epsilon')\ell q^{-m}$, where $\epsilon'=\epsilon(1+\epsilon)^{-1}$.
Let $L$ be an $m$-flat.
By \cref{eq:simplifiedUpperChernoff},
\[\Pr[|L \cap X| \geq \ell] \leq \exp\left(\frac{-\epsilon^2 \theta q^m}{(2+\epsilon)} \right) = \exp\left(\frac{-\epsilon^2 \ell}{(2+\epsilon)(1+\epsilon)} \right).\]
Summing over all $m$-flats, the expected value of $Y$ is
\[\ee[Y] = \ee\left [\sum_{L \in \mathcal{L}}  \mathbf{1}(|L \cap X| \geq \ell)\right] = \sum_{L \in \mathcal{L}} \ee[ \mathbf{1}(|L \cap X| \geq \ell)] \leq q^{O((n-m)m)} \exp\left(\frac{-\epsilon^2 \ell}{(2+\epsilon)(1+\epsilon)} \right). \]
Since $\ell = \omega((n-m)m \log(q))$, this implies that $\ee[Y] \rightarrow 0$.

By Markov's inequality, the expectation of $Y$ is an upper bound on the probability of $\Gamma$:
\[\Pr[\Gamma] = \Pr[Y\geq 1] \leq \ee[Y] \rightarrow 0.\]

Now for $\epsilon \in (0,1)$, suppose that $\th = \ell q^{-m}(1-\epsilon)^{-1} = (1+\epsilon')\ell q^{-m}$, where $\epsilon' = \epsilon(1-\epsilon)^{-1}$.
Let $L$ be an $m$-flat.
By \cref{eq:lowerChernoff},
\[\Pr[|L \cap X| < \ell] \leq \exp\left(\frac{-\epsilon^2 \theta q^{m}}{2} \right) =  \exp\left(\frac{-\epsilon^2 \ell}{2(1-\epsilon)} \right).\]

Since $\ell = \omega(1)$, we have that $\Pr[|L \cap X| < \ell] \rightarrow 0$.
Hence,
\[\Pr[Y \geq 1] \geq \Pr[|L \cap X| \geq \ell] = 1 - \Pr[|L \cap X| < \ell] \rightarrow 1. \]

\end{proof}

\mflat*

The proof of \cref{m-flat} has an inductive step that relies on the following corollary, which uses a slightly different probabilistic model.

\begin{cor}\label{th:mFlatMNModel}
Suppose $m \geq 2$ is bounded and $q$ is a prime power with $q \rightarrow \infty$.
Let $\ell \geq m+2$ be a constant integer, and let $X$ be a uniformly random set of $\ell$ points in $\FF_q^m$.
Let $\Lambda = \Lambda(m,\ell)$ be the event that there is an $(m+1)$-rich $(m-1)$-flat with respect to $X$.
Then, $\Pr[\Lambda] \rightarrow 0$.
\end{cor}
\begin{proof}
Let $\theta = \ell q^{-m}$, let $X'$ be a random subset of $\FF_q^m$ such that each element of $\FF_q^m$ is in $X$ with probability $\theta$, independently of all other elements.
Let $\Gamma$ be the event that there is an $(m+1)$-rich $(m-1)$-flat with respect to $X'$.
By \cref{m-flat} (or \cref{l-col} in the case $m=2$), 
$\Gamma$ has a threshold function $t = q^{-(m - 1 + m/(m+1))}$.
Since $t > \theta$ for $q$ sufficiently large, we have $\Pr[\Gamma] \rightarrow 0$.

We use a standard coupling argument to show that $\Pr[\Lambda] = O(\Pr[\Gamma])$.
In more detail,
\begin{align*}
    \Pr[\Gamma] &= \sum_{k=0}^{q^m} \Pr[\Gamma \, \vert \, |X'|=k] \Pr[|X'| = k] \\
    &= \sum_{k=0}^{q^m} \Pr[\Lambda(m,k)] \Pr[|X'| = k] \\
    &\geq \Pr[\Lambda] \Pr[|X'| = \ell].
\end{align*}
Using the inequality $(1-x) < e^x$ to bound $\Pr[|X'|=\ell]$:
\begin{align*}
    \Pr[|X'| = \ell] &= \binom{q^m}{\ell} \theta^\ell (1-\theta)^{q^m - \ell} \\
    &= (1-o(1)) \frac{\ell^\ell}{\ell!} (1-\ell q^{-m})^{q^m} \\
    &= O(1-\ell q^{-m})^{q^m} \\
    &= O(e^{-\ell}) = O(1).
\end{align*}
Consequently, we have $\Pr[\Lambda] = O(\Pr[\Gamma])$, which implies the result.
\end{proof}

\begin{proof}[Proof of \textup{\cref{m-flat}}]
Let $\mathcal{G}$ be the family of $\ell$-subsets $S$ of $\F$ such that no $m-1$ flat contains more than $m$ points of $S$.
Let $\G_\mathcal{G} = \G(\mathcal{G}; \th)$ be the event that $X$ contains an element of $\mathcal{G}$.

Let $t_{\f}$ be a threshold function for $\G$.
Since every element of $\mathcal{G}$ is also an element of $\f$, there is a threshold function $t_{\mathcal{G}}$ for $\G_\mathcal{G}$ such that $t_{\mathcal{G}} > t_{\f}$.
Furthermore, one can show that \textup{Pr[$\G$]}$\to 0$ if $\theta=o(|\f|^{-1/\ell})$ using Markov's inequality, similarly to the argument used in the proof of \cref{new}. 
Hence, $t_{\f}=\Omega(|\f|^{-1/\ell})$.

In what follows, we show that
\begin{enumerate}
    \item $t_{\mathcal{G}} = O(|\mathcal{G}|^{-1/\ell})$, and
    \item $|\mathcal{G}| = \Omega(|\f|)$.
\end{enumerate}
Combined with the observations in the previous paragraph, this leads to the following sequence of inequalities:
\begin{equation}\label{rel}
\Omega(|\f|^{-1/\ell})= t_{\f} \leq t_{\mathcal{G}} = O(|\mathcal{G}|^{-1/\ell}) = O(|\f|^{-1/\ell}).
\end{equation}
Therefore, $t_{\f}=|\f|^{-1/\ell}$ is a threshold function for $\G$.

We use induction to show that $|\mathcal{G}| = (1-o(1))|\f|$.
Suppose that $S$ is a uniformly random set of $\ell$ points in $\F$.
Since $\mathcal{G} \subset \f$, we have
\[\Pr[S \in \mathcal{G} | S \in \f] = \frac{|\mathcal{G}|}{|\f|}. \]
$\Pr[S \notin \mathcal{G} | S \in \f]$ is exactly the probability that a uniformly random set of $\ell$ points in $\mathbb{F}_q^m$ contains a subset of size $m+1$ that is contained in an $(m-1)$-flat.
By \cref{th:mFlatMNModel}, $\Pr[S \notin \mathcal{G} | S \in \f] \rightarrow 0$, and hence $|\mathcal{G}| \, |\f|^{-1} \rightarrow 1$, as desired.

We now apply \cref{thm:revision} for $\mathcal{G}$. Note that the number of $m$-planes in $\mathbb{F}_{q}^{n}$ containing $t$ fixed points in general position is $\binom{n-t+1}{m-t+1}_{q}$, where $1 \leq t \leq m$.
Thus, the number of $\ell$-rich $m$-flats in general position containing $t$ fixed points is bounded by
\[\binom{n-t+1}{m-t+1}_{q}\binom{q^{m}-t}{\ell-t}=(1+o(1))q^{(n-t+1)(m-t+1)-(m-t+1)^{2}+m(\ell-t)},\]
which is also bounded by 
\begin{equation}\label{Gsize}
    |\mathcal{G}|^{1-t/\ell}=(1+o(1))q^{((n-m)(m+1)+m\ell)(1-t/\ell)}
\end{equation}
when $q$ is large enough since $\ell>m$. For $m+1 \leq t \leq \ell -1$, the number of $\ell$-rich $m$-flats in general position containing $t$ fixed points is bounded by $\binom{q^{m}-t}{\ell -t}$, which is also bounded by \cref{Gsize}. This verifies the assumption of \cref{thm:revision}, showing that the event $\G_{\mathcal{G}}=\G(\mathcal{G};\th)$ has a coarse threshold function $t_{\mathcal{G}}=|\mathcal{G}|^{-1/\ell}=q^{-((n-m)(m+1)+m\ell)/\ell}$ as $q\to \infty$. This completes the proof by \cref{rel}.
\end{proof}

\section{Contained flats and blocking sets}

In this section, we prove \cref{m-planes} and \cref{th:blockingSetThreshold}.
These both follow as corollaries to the following theorem on the probability that a random set of points contains an $m$-flat.

\begin{thm}\label{th:strongM-plane}
Let $m,n$ be integers and $q$ a prime power with $n+q \rightarrow \infty$ and $m$ constant with $2m<n$.
Let $t \rightarrow \infty$.
Let $\f$ be the family of $m$-flats in $\F$, and let $\ell=q^m$ be the number of points in an $m$-flat.
Let $X$ be a random subset of $\F$ such that each point of $\F$ is in $X$ with probability $\th$, independently of all other points, and let $\G = \G[\th;m]$ be the event that some $m$-flat is entirely contained in $X$. 
Then,
\begin{enumerate}
    \item if $\th < (|\f|t)^{-1/\ell}$, then $\Pr[\G] \rightarrow 0$, and
    \item if $\th > (|\f|t^{-1})^{-1/\ell}$, then $\Pr[\G] \rightarrow 1$.
\end{enumerate}
\end{thm}

\begin{proof}
First suppose that $\th < (|\f|t)^{-1/\ell}$.
Let $Y$ be the number of $m$-flats contained in $X$.
Using Markov's inequality,
\[\Pr[\G] = \Pr[Y \geq 1] \leq \ee[Y] = \th^\ell|\f| < t^{-1} \rightarrow 0,\]
as desired.

Now suppose that $\th > (|\f|t^{-1})^{-1/\ell}$, and let $\th'=(|\f|s^{-1})^{-1/\ell} \geq \th$, where $s \leq t$ and $s \rightarrow \infty$ is a function we will choose later.
Since $\Pr[\G[\th;m]]$ increases monotonically with $\th$, it is enough to show that $\Pr[\G[\th';m]] \rightarrow 1$.

To apply \cref{th:2MomentThreshold}, we should bound $I_k$.
This is non-zero only when $k$ is the number of points contained in some $d$-flat, for $0 \leq d < m$.
For $k=q^d$, we have that $I_k$ is bounded from above by the number of $d$-flats times the square of the number of $m$-flats that contain a fixed $d$-flat.
Hence, using \cref{eq:qBinBound}, for integers $d$ with $0 \leq d < m$,
\begin{equation}\label{eq:Iqd}
|I_{q^d}| \leq 4^3q^{(n-d)(d+1)+2(n-m)(m-d)},\end{equation}
and $|I_k| = 0$ otherwise.
Note, if $n+d \leq 2m$, then this would be weaker than the trivial bound of $|I_k| < |\f|^2$.
This is avoided by the assumption that $2m < n$.

Now we apply \cref{th:2MomentThreshold} with $c'=s^{-1}$ and $c=s^{-1/\ell}$.
Since $m$ is bounded, it is enough to show that
\[(|\f|s^{-1})^{k/\ell} |I_k| = o(|\f|^2 s^{-1}) \]
for each $k = q^d$ with $0 \leq d < m$.
Equivalently,
\begin{equation}(|\f|s^{-1})^{k/\ell} s = o(|\f|^2 \, |I_k|^{-1}). \end{equation}
Using \cref{eq:Iqd}, we have
 $|\f|^2 \, |I_k|^{-1} = \Omega(q^{(n-2m+d)(d+1)})$,
so it is enough to show 
\begin{equation}\label{eq:f2IkBlocking}
(|\f|s^{-1})^{q^{d-m}} s = o(q^{(n-2m+d)(d+1)})
 \end{equation}
 as $n+q \rightarrow \infty$.
 
Suppose that $n$ is bounded, and hence by assumption $q \rightarrow \infty$.
Choose $s \leq t$ such that $s = O(q^{n-2m-\epsilon})$  for some $0 < \epsilon < 1$ and $s \rightarrow \infty$;  this is possible since $2m < n$.
Now, \cref{eq:f2IkBlocking} is
\[q^{(nm - m^2 + m + \epsilon)q^{d-m} + n-2m-\epsilon} = o(q^{(n-2m+d)(d+1)}), \]
which holds for $0 \leq d < m$ and $q$ sufficiently large.
Hence, \cref{th:2MomentThreshold} implies that $\Pr[\G] \geq 1 - 2 s^{-1} \rightarrow 1$, as desired.
 
Otherwise, $n \rightarrow \infty$, and it is enough to show that 
\[q^{(n-2m+d)(d+1)}|\f|^{-q^{d-m}} = \Omega(q^{(n-2m+d)(d+1) - q^{d-m}(n-m)(m+1)})\]
is not bounded for any $d$, so that \cref{eq:f2IkBlocking} is satisfied for some $s \rightarrow \infty$.
Dropping constant terms in the exponent, we need $n((d+1) - q^{d-m}(m+1)) > 0$ for all $0 \leq d < m$.
This holds, except in the case $m=1$ and $q=2$.
In this case, $\G$ is just the event that $X$ contains at least $2$ points.
\end{proof}

\mplanes*

\begin{proof}
We have
\[(t|\f|)^{-1/\ell} = \Theta(t^{-q^{-m}}q^{-(n-m)(m+1)q^{-m}}) = \Theta(t^{-q^{-m}}q^{-n(m+1)q^{-m}}). \]

If $q$ is bounded and
\[\th = o(q^{-n(m+1)q^{-m}}) = \omega(1)^{-q^{-m}}q^{-n(m+1)q^{-m}}, \]
then \cref{th:strongM-plane} implies that $\Pr[\G] \rightarrow 0$.

If we have the three conditions $q \rightarrow \infty$, $\epsilon > 0$ is constant, and
\[\th= (1-\epsilon)q^{-n(m+1)q^{-m}} = ((1-\epsilon)^{-q^{m}})^{-q^{-m}} q^{-n(m+1)q^{-m}} = \omega(1)^{-q^{-m}} q^{-n(m+1)q^{-m}}, \]
then \cref{th:strongM-plane} implies that $\Pr[\G] \rightarrow 0$.

The other bounds follow similarly.
\end{proof}

\blockingSetThreshold*

\begin{proof}
Note that $X$ is an $m$-blocking set if and only if $X^c$ does not contain an $m$-flat.
We will use the estimate $|\f| = (1+o(1))q^{(n-m)(m+1)} = q^{\Theta(n)}$, as well as $1-x < e^{-x}$ valid for $x>0$, and $1-x > 1-2x+2x^2 > e^{-2x}$, valid for $0<x<1/2$.

Suppose that $\th = \omega(nq^{-m}\log(q))$.
Then,
\[1-\th < e^{-\th}  = e^{-\omega(n q^{-m} \log(q))} = q^{-\omega(n q^{-m})} = |\f|^{-\omega(1/\ell)} = (t|\f|)^{-1/\ell}\]
for some $t \rightarrow \infty$.
Hence, by \cref{th:strongM-plane}, the probability that $X^c$ contains an $m$-flat approaches $0$, and so the probability that $X$ is an $m$-blocking set approaches $1$.

Now suppose that $\th = o(nq^{-m}\log(q))$.
Then, for $q$ large enough,
\[1-\th > e^{-2\th} =  e^{-o(n q^{-m} \log(q))} = q^{-o(n q^{-m})} = |\f|^{-o(1/\ell)} = (t^{-1}|\f|)^{-1/\ell}\]
for some $t \rightarrow \infty$, as desired. 
Note that we used the assumption that $n=o(q^m(\log q)^{-1})$, and hence $\th = o(1)$, in the first inequality.
\end{proof}

\section{Point-flat incidences}

This section contains the proof of \cref{th:incidenceThreshold}.

\incidenceThreshold*

\begin{proof}
Let $\f$ be the set of pairs $(x,y)$ of a point $x$ and an $m$-flat $y$ in $\F$ so that $x \in y$.
Then $|\f| = \Theta(q^{n+(n-m)m}) = \Theta(q^{(n-m)(m+1) + m})$. 
To apply \cref{new}, we only need to show that
\begin{equation}
\label{eq:I1incidences}
|I_1| = O(|\f|^{3/2}).\end{equation}
A pair in $I_1$ either shares a point, or an $m$-flat, so $|I_1| = \Theta(q^n q^{2(n-m)m}) + \Theta(q^{(n-m)(m+1)}q^{2m}) = \Theta(q^{n + 2(n-m)m})$.
Combining this with the bound on $|\f|$ and rearranging, we have that \cref{eq:I1incidences} holds if $(n-m)m \leq n$.
This holds if $m = n-1$, or $m=1$, or $n=4$ and $m=2$.
For the case $n=4$, $m=2$, we have $(1/2)(n+(n-m)m) = n$, so we have the alternate threshold $q^{-n}$. 

If $\th = o(q^{-n})$, then $\Pr[|X| > 0] \rightarrow 0$, and hence $\Pr[\G] \rightarrow 0$.
Now suppose that $\th = \omega(q^{-n})$.
We will use \cref{th:2MomentThreshold} with $c=q^n|\f|^{-1/2}$ and $c' \rightarrow 0$.
Note that $c \rightarrow 0$ if and only if $|\f| = o(q^{2n})$, or equivalently $n + (n-m)m > 2n$.
This holds exactly when $m \neq 1$ and $m \neq n-1$.
With these parameters, the condition of \cref{th:2MomentThreshold} is that $|I_1| \leq o(|\f|^2 q^{-n})$.
From \cref{eq:I1incidences}, this holds exactly when $|\f| = o(q^{2n})$, which we have already checked.
\end{proof}

\section{Acknowledgements}
Jeong Han Kim was partially supported by National Research Foundation of Korea (NRF) Grants funded by the Korean Government (MSIP) (NRF-2016R1A5A1008055 \& 2017R1E1A1A0307070114) and by a KIAS Individual Grant(CG046002) at Korea Institute for Advanced Study. B. Lund was supported by the Institute for Basic Science (IBS-R029-C1). T. Pham was supported by the National Foundation for Science and Technology Development (NAFOSTED) Project 101.99-2021.09. S. Yoo was supported by the KIAS Individual Grant (CG082701) at Korea Institute for Advanced Study.

T. Pham would like to thank the VIASM for the hospitality and for the excellent working condition.
The authors are grateful to the reviewers for valuable comments.

\end{document}